\documentclass[a4paper]{article}
\usepackage{amsmath}
\usepackage{amsfonts}
\usepackage{graphicx}
\usepackage{caption}
\usepackage{comment}
\usepackage{pstricks, pst-node}
\usepackage{epstopdf}
\usepackage[caption=false]{subfig}
\usepackage[a4paper]{geometry}
\makeatother
\newtheorem{proposition}{Proposition}[section]
\newtheorem{theorem}{Theorem}[section]
\newenvironment{proof}[1][Proof]{\textbf{#1.} }{\ \rule{0.5em}{0.5em}}
\newcommand{\pai}{\left(}
\newcommand{\pad}{\right)}
\newcommand{\ci}{\left[}
\newcommand{\cd}{\right]}
\newcommand{\p}{\mathbb{P}}

\title{Mathematical modelling of the performance of a student in non-collaborative and non-presential learning}
\author{A.R. Sagaceta-Mej\'ia$^{(1)}$, J. A. Fres\'an-Figueroa$^{(2)}$ and E.M. Mart\'in-Gonz\'alez.$^{(3)}$\\
$^{(1)}$ Departamento de Física y Matemáticas, Universidad Iberoamericana, México\\
e-mail: alma.sagaceta@ibero.mx, \\
$^{(2)}$ Departamento de Matemáticas Aplicadas y Sistemas, \\Universidad Autónoma Metropolitana, Unidad Cuajimalpa, México, \\
e-mail: jfresan@cua.uam.mx,\\
$^{(3)}$ Departamento de Matem\'aticas, Universidad de Guanajuato, M\'exico\\
e-mail:  ehyter.martin@ugto.mx}

\begin{document}

\maketitle

\begin{abstract}
	In this paper we propose a model to study the appropriation of knowledge of one student in a non-collaborative online class. We formulate a stochastic model based on the quality of the teacher's class and the affinity of the student to understand the sessions, under the assumption that previous sessions have some influence in the understanding of the next sessions. This assumption implies that the process is not even a Markov process. This kind of situation appears in seminars and classes with many different sessions. We derive some recursive expressions for the distribution of the number of sessions that the student comprehends. Furthermore, we study the convergence of this distribution and study the speed of this convergence through some numerical examples.
\end{abstract}

\section{Introduction}

Simple methods from physics and mathematics have been recently adopted to model, as mathematical metaphors, a wide range of social phenomena and social systems \cite{cajueiro,klemm,li_2018,naka,ucam,stauffer}. However, to the best of our knowledge, this kind of models have not been comprehensively used to study the way students learn in a certain class. How the behaviour of students affects how well they learn is an old question and has been widely studied in several contexts since people started to concern about teaching and pedagogy \cite{Dufresne,lee2007,McCroskey,Richmond}. Currently, in the context of the \textit{forced} transition to online education due to the worldwide health alert, providing an answer to this important issue and other closely related problems could legitimize or give birth to new learning techniques. After this epidemic it can be expected that more capital will be invested to online teaching and blended learning. Even though moving to a total online or blended learning model will be a long process, it is important to  study the intrinsic behaviour of students to provide better answers to the challenges ahead.

In this work we proffer to answer the following question:
\begin{itemize}
	\item \textit{How does a student, non-interacting with others, understands according to their affinity to the sessions, during a course?}
\end{itemize}

Although general behaviors could be modeled as particles or agents (students) that interact and exchange "something" (knowledge), in the current environment of social distancing; it is a known fact that many students have resented this way of living in "solitude" particularly in the classroom. This way of being in solitude presents new opportunities and challenges when trying to measure how the student understand a class session. In the present work we wish to model the possible appropriation of knowledge, whose easiness varies across the  student's understanding throughout the course and the complexity of each session if they do not interact with their classmates.

We assume that the course is presented in a number of sessions such that session $j$ helps the student's understanding of session $j+1$. This assumption implies that the given model does not rely on a Markov process \cite{Markov}, because session $j$ depends strongly on all the $j-1$ previous sessions.

This work is organized as follows: A general description of the model is given in  Section 2. The mathematical manipulation of the model and some results are presented in Section 3. Afterwards, we study some particular cases in Section 4 and present some numerical examples of the mode. Finally, conclusions and future work appear in Section 5.

\section{Description of the Model}

We consider the situation in which the course consists of $n$ sessions and the learning process of each student in the same course is independent of the other students.

Under this situation, we assume that the lecturer teaches each session according to a measurable parameter $q$, which represents the quality of the sessions. Therefore we refer to $q$ as the \textit{quality parameter} of each session. In this work we only consider the case when this parameter remains constant along the whole course.

The event in which the student understands the first session has a probability given by $\overline{F}(1-q),$ for a certain probability distribution $F$, where $\overline{F}:=1-F$.

From the second session until the end of the course, if the student \textit{understood} session $j$, they understand the next one with probability $\overline{F}_{j+1}(1-q):=\overline{F}_j(1-q-\varepsilon)$. Here $F_1$ refers to some initial distribution $F$ and each $F_j$ for $j\geq 2$ is constructed conditioned on the result of all the previous $j-1$ sessions.

Similarly if the student \textit{did not understand} session $j$, they understand the next one with probability $\overline{F}_{j+1}(1-q):=\overline{F}_j(1-q+\varepsilon)$.

The parameter $\varepsilon$ is assumed to be positive and fixed during the entire course and it reflects how the comprehension of the content of a session influences the next session. Hence we call $\varepsilon$ the \textit{dependence} parameter. In some courses this dependence parameter will be relatively big in comparison with $n$, like in a Calculus class.  But in many other cases, $\varepsilon$ will be relatively small compared to $n$, like in seminaries, or panoramic courses.

We wish to avoid the situation when the student understands the last sessions of the course with probability 1, due to the cumulative effect of $\varepsilon$ after some point, therefore we assume that $\varepsilon$ is $o(f(n))$,  for a properly chosen function $f$ depending on the total number of sessions $n$.

In the following sections we manipulate this model to obtain some results related to the distribution of the number of sessions that the student understands along the course.

\section{Main results}

We let $Y(j)$ be Bernoulli random variables defined as follows:
For session 1, $Y(1)$ is Bernoulli with parameter $\overline{F}(1-q)$ for the given initial continuous distribution $F$.

For session 2, by the Law of Total Probability we have:

\begin{align}
	\mathbb{P}\left[Y(2)=1\right]&=\mathbb{P}\left[Y(2)=1|Y(1)=1\right]\overline{F}(1-q)+\mathbb{P}\left[Y(2)=1|Y(1)=0\right]F(1-q)\label{Ytema2}
\end{align}

Given that the student understood session 1, the probability that they understand session 2 is given by

$$\mathbb{P}\left[Y(2)=1|Y(1)=1\right]=\overline{F}(1-q-\varepsilon).$$

Similarly, if the student did not understand session 1, the probability that they understand session 2 is

$$\mathbb{P}\left[Y(2)=1|Y(1)=0\right]=\overline{F}(1-q-\varepsilon).$$

It follows that (\ref{Ytema2}) is equivalent to

\begin{align}
	\mathbb{P}\left[Y (2)=1\right]&=\overline{F}(1-q-\varepsilon)\overline{F}(1-q)+\overline{F}(1-q+\varepsilon)F(1-q).
\end{align}

For the general setting, we denote by $F_j$ the probability distribution such that

$$\overline{F}_j(x)=\overline{F}_{j-1}(x-\varepsilon)p_{j-1}+\overline{F}_{j-1}(x+\varepsilon)\left(1-p_{j-1}\right),$$

provided that $\overline{F}_{j-1}(x-\varepsilon)$ and $\overline{F}_{j-1}(x+\varepsilon)$ do not equal zero or one. In the equation above, $p_k:=\p\ci Y(k)=1\cd$ for $k>1$ with $p_1=\overline{F}(1-q)$ and $F_1:=F$.

Using this notation we obtain

\begin{equation}\label{pj_general}
	p_j=\overline{F}_{j-1}\big(1-q-\varepsilon\big)p_{j-1}+\overline{F}_{j-1}\big(1-q+\varepsilon\big)(1-p_{j-1}).
\end{equation}

We derive a recursive expression for the probabilities $\{p_m,1<m\leq n\}$, which is given in the following result.

\begin{theorem} \label{Teo1}
	Let $n$ and $\varepsilon$ be such that $\left(n-1\right)\varepsilon<\min\left\{ 1-q,q\right\} $,
	then the general expression for $\{p_m,1<m\leq n\}$ reads

	$$p_{m}  =\overline{F}_{1}\left(1-q\right)\prod_{j=1}^{m-1}\left[\overline{F}_{j}\left(1-q-\varepsilon\right)-\overline{F}_{j}\left(1-q+\varepsilon\right)\right]+\sum_{i=1}^{m-1}\overline{F}_{i}\left(1-q+\varepsilon\right)\prod_{j=i+1}^{m-1}\left[\overline{F}_{j}\left(1-q-\varepsilon\right)-\overline{F}_{j}\left(1-q+\varepsilon\right)\right].$$
\end{theorem}
\begin{proof}
	The result holds for $m=1$ by the definition of $p_{1}$. We proceed by induction, assuming that for an integer
	$k\geq1$,

	$$p_{k}  =\overline{F}_{1}\left(1-q\right)\prod_{j=1}^{k-1}\left[\overline{F}_{j}\left(1-q-\varepsilon\right)-\overline{F}_{j}\left(1-q+\varepsilon\right)\right]+\sum_{i=1}^{k-1}\overline{F}_{i}\left(1-q+\varepsilon\right)\prod_{j=i+1}^{k-1}\left[\overline{F}_{j}\left(1-q-\varepsilon\right)-\overline{F}_{j}\left(1-q+\varepsilon\right)\right].$$
	
	By equation (\ref{pj_general})
	$$p_{k+1}  =\overline{F}_{k}\left(1-q-\varepsilon\right)p_{k}+\overline{F}_{k}\left(1-q+\varepsilon\right)\left(1-p_{k}\right)\\
	=p_{k}\left[\overline{F}_{k}\left(1-q-\varepsilon\right)-\overline{F}_{k}\left(1-q+\varepsilon\right)\right]+\overline{F}_{k}\left(1-q+\varepsilon\right),$$
	
	hence we obtain from the induction hypothesis 
	
	\begin{align*}
		p_{k+1} & =\left[\overline{F}_{k}\left(1-q-\varepsilon\right)-\overline{F}_{k}\left(1-q+\varepsilon\right)\right]\overline{F}_{1}\left(1-q\right)\prod_{j=1}^{k-1}\left[\overline{F}_{j}\left(1-q-\varepsilon\right)-\overline{F}_{j}\left(1-q+\varepsilon\right)\right]\\
		& +\left[\overline{F}_{k}\left(1-q-\varepsilon\right)-\overline{F}_{k}\left(1-q+\varepsilon\right)\right]\sum_{i=1}^{k-1}\overline{F}_{i}\left(1-q+\varepsilon\right)\prod_{j=i+1}^{k-1}\left[\overline{F}_{j}\left(1-q-\varepsilon\right)-\overline{F}_{j}\left(1-q+\varepsilon\right)\right]\\
		& +\overline{F}_{k}\left(1-q+\varepsilon\right)\\
		& =\overline{F}_{1}\left(1-q\right)\prod_{j=1}^{k}\left[\overline{F}_{j}\left(1-q-\varepsilon\right)-\overline{F}_{j}\left(1-q+\varepsilon\right)\right]\\
		& +\sum_{i=1}^{k}\overline{F}_{i}\left(1-q+\varepsilon\right)\prod_{j=i+1}^{k}\left[\overline{F}_{j}\left(1-q-\varepsilon\right)-\overline{F}_{j}\left(1-q+\varepsilon\right)\right].
	\end{align*}
	
	The result now follows.
\end{proof}
\bigskip

\noindent From this point on, we drop the notation $F_1$ and write $F$ instead.
\bigskip

We let $B_{ n}$ denote the number of sessions, not necessarily consecutive, from a total of $n$ that
the student understood. We are interested in the probability function of $B_n$, $\left\{\p\ci B_{ n}=k \cd,0\leq k\leq n\right\}$, for which we consider the following particular cases.

\begin{itemize}
	\item \textit{Case 1: $n=3$ and $k=0$}. 
	
	This is the case when the student understands 0 sessions out of 3. The probability reads:
	$$
	\p\ci B_{3}=0\cd =F\left(1-q\right)F\left(1-q+\varepsilon\right)F\left(1-q+2\varepsilon\right)
	$$
	
	\item \textit{Case 2: $n=3$ and $k=1$}.
	
	If student understand only one session (i.e. $k=1$), the probability is the sum of the following cases.
	
	\begin{enumerate}
		\item[2.1] The student understands only the first session $$\overline{F}\left(1-q\right)F\left(1-q-\varepsilon\right)F\left(1-q\right).$$
		\item[2.2] Student $i$ understands only the second session $$F\left(1-q\right)\overline{F}\left(1-q+\varepsilon\right)F\left(1-q\right).$$
		\item[2.3] And Student $i$ understands only the third session $$F\left(1-q\right)F\left(1-q+\varepsilon\right)\overline{F}\left(1-q+2\varepsilon\right).$$
	\end{enumerate}
	
	Note that the probabilities in cases 2.1 and 2.2 equal $\p\ci B_{2}=0 \cd F\left(1-q\right)$
	and the probability of case 2.3 corresponds to $\p\ci B_{2}=1\cd \overline{F}\left(1-q+2\varepsilon\right)$. Hence
	
	$$\p\ci B_{3}=1\cd = \p\ci B_{2}=0 \cd F\left(1-q\right)+\p\ci B_{2}=0\cd F\left(1-q\right).$$
	
	\item    \textit{Case 3: $n=3$ and $k=2$}.
	If student understand two sessions (i.e. $k=2$), the probability is the sum of the following cases:
	\begin{enumerate}
		\item[3.1] Student $i$ understands the first and second sessions but not the third
		$$\overline{F}\left(1-q\right)\overline{F}\left(1-q-\varepsilon\right)F\left(1-q-2\varepsilon\right).$$
		\item[3.2] Student $i$ understands the second and third sessions but not the first
		$$F\left(1-q\right)\overline{F}\left(1-q+\varepsilon\right)\overline{F}\left(1-q\right).$$
		\item[3.3] Student $i$ understands the first and third sessions but not the second
		$$\overline{F}\left(1-q\right)F\left(1-q-\varepsilon\right)\overline{F}\left(1-q\right).$$
	\end{enumerate}
	Note that cases 2 and 3 correspond to $\p\ci B_{2}=1\cd \overline{F}\left(1-q\right)$
	and the first case equals $\p\ci B_{2}=2\cd F\left(1-q-2\varepsilon\right)$. Hence
	$$\p\ci B_3=2\cd=\p\ci B_{2}=1\cd\overline{F}\left(1-q\right)+\p\ci B_{2}=2\cd F\left(1-q-2\varepsilon\right).$$
	
	\item\textit{ Case 4: $n=3$ and $k=3$}.
	
	This is the case when student $i$ understands all sessions. This probability is given by
	$$
	\p\ci B_{3}=3\cd=\overline{F}\left(1-q\right)\overline{F}\left(1-q-\varepsilon\right)\overline{F}\left(1-q-2\varepsilon\right).
	$$
	
\end{itemize}

\noindent The recursive behaviour observed in the probability function of $B_3$ is generalized in the following theorem.

\begin{theorem}
	\label{theo:A}
	
	Let $n\geq 3$ be an integer. The probability function of the random variable $B_n$ satisfies the following relations.
	
	\begin{enumerate}
		\item $\p\ci B_{ n}=0\cd=\p\ci B_{n-1}=0\cd F\left(1-q+\left(n-1\right)\varepsilon\right),$\\
		\item $\p\ci B_{ n}=n\cd=\p\ci B_{n-1}=n-1\cd \overline{F}\left(1-q-\left(n-1\right)\varepsilon\right),$\\
		\item $\p\ci B_{ n}=k\cd=\p\ci B_{n-1}=k\cd F\left(1-q+\left(n-1-2k\right)\varepsilon\right)
		+\p\ci B_{n-1}=k-1\cd \overline{F}\left(1-q+\left(n-1-2\left(k-1\right)\right)\varepsilon\right).$
	\end{enumerate}
\end{theorem}
\begin{proof}
	\begin{enumerate}
		\item If the student has not understand the first $n-1$ sessions from a total of $n$, it follows from the construction of the model that the student's parameter for understanding the $nth$-session becomes $1-q+(n-2)\varepsilon$. Hence
		\begin{align*}
			\p\ci B_{ n}=0\cd&=\p\ci B_{n-1}=0,Y (n)=0\cd=
			\p\ci Y (n)=0 \mid B_{n-1}=0\cd \p\ci B_{n-1}=0\cd \\
			&=F(1-q+(n-1)\varepsilon)\p\ci B_{n-1}=0\cd.
		\end{align*}
		\item Let $A_n$ denote the number of sessions that the student has not understood from a total of $n$. Then the event $\{B_n=n\}$ is the same as $\{A_n=0\}$ and hence
		$\p\ci B_{ n}=n\cd=\p\ci A_{ n}=0\cd$. Now the result in 1 yields
		\begin{align*}
			\p\ci B_{ n}=n\cd &=\p\ci A_{n-1}=0\cd\ci 1-F\pai 1-q+\pai n-1\pad\varepsilon\pad\cd =\p\ci B_{n-1}=n-1\cd \overline{F}(1-q+(n-1)\varepsilon).
		\end{align*}
		\item Let $U_k(n):=\left\{(x_1,\dots,x_n)\in \{0,1\}^n:x_1+\dots+x_n=k\right\}$,     then
		\small
		\begin{align}
			\p\ci B_{ n}=k\cd        &=\p\ci B_{n-1}=k,Y (n)=0\cd+P\ci B_{n-1}=k-1,Y (n)=1\cd\nonumber \\
			&=\sum\limits_{(x_1,\dots,x_{n-1})\in U_k(n-1)}\p\ci Y (1)=x_1,\dots,Y (n-1)=x_{n-1},Y (n)=0\cd\nonumber \\
			&+\sum\limits_{(x_1,\dots,x_{n-1})\in U_{k-1}(n-1)}\p\ci Y (1)=x_1,\dots,Y (n-1)=x_{n-1},Y (n)=1\cd\nonumber  \\
			&=\sum\limits_{(x_1,\dots,x_{n-1})\in U_k(n-1)}\p\ci Y (n)=0 \mid Y (1)=x_1,\dots,Y (n-1)=x_{n-1}\cd\p\ci  Y (1)=x_1,\dots,Y (n-1)=x_{n-1}\cd\nonumber \\
			&+\sum\limits_{(x_1,\dots,x_{n-1})\in U_{k-1}(n-1)}\p\ci Y (n)=1 \mid Y (1)=x_1,\dots,Y (n-1)=x_{n-1}\cd \p\ci  Y (1)=x_1,\dots,Y (n-1)=x_{n-1}\cd. \label{probafea0}
		\end{align}
		
		\normalsize
		\noindent    Note that, given the configuration $Y (1)=x_1,\dots,Y (n-1)=x_{n-1}$ in which the student understood exactly $k$ sessions, we have added $k$ times $\varepsilon$  to the quality parameter $q$. Moreover, the maximum number of times we may add or subtract $\varepsilon$ in a total of $n$ sessions equals $n-1$, since we do not add or subtract anything in session 1. Hence if the student understood $k$ of $n$ sessions, they did not understand $n-k$ and we have subtracted $n-1-k$ times $\varepsilon$. This means that understanding session $n$ depends on the parameter
		
		$$q+k\varepsilon-(n-1-k)\varepsilon=q-(n-1-2k)\varepsilon.$$
		
		\noindent    It follows that the student does not understand session $n$ with probability $F\pai 1-q-(n-1-2k)\varepsilon\pad. $ Since this holds for any given configuration $Y (1)=x_1,\dots,Y (n-1)=x_{n-1}$, in which the student has understood exactly $k$ sessions, we obtain
		\begin{align}
			&\sum\limits_{(x_1,\dots,x_{n-1})\in U_k(n-1)}\p\ci Y (n)=0 \mid Y (1)=x_1,\dots,Y (n-1)=x_{n-1}\cd\p\ci  Y (1)=x_1,\dots,Y (n-1)=x_{n-1}\cd\nonumber \\
			&=F\pai1-q-(n-1-2k)\varepsilon\pad\sum\limits_{(x_1,\dots,x_{n-1})\in U_k(n-1)}\p\ci  Y (1)=x_1,\dots,Y (n-1)=x_{n-1}\cd\nonumber \\
			&=F\pai1-q-(n-1-2k)\varepsilon\pad\p\ci B_{n-1}=k\cd.\label{probafea1}
		\end{align}
		\noindent   Analogously
		\begin{align}
			&\sum\limits_{(x_1,\dots,x_{n-1})\in U_{k-1}(n-1)}\p\ci Y (n)=1 \mid Y (1)=x_1,\dots,Y (n-1)=x_{n-1}\cd\p\ci  Y (1)=x_1,\dots,Y (n-1)=x_{n-1}\cd\nonumber \\
			&=\overline{F}\pai1-q-(n-1-2(k-1))\varepsilon\pad\p\ci B_{n-1}=k-1\cd.\label{probafea1.2}
		\end{align}
		\noindent The result follows by substituting equations (\ref{probafea1}) and (\ref{probafea1.2}) in equation (\ref{probafea0}).
	\end{enumerate}
\end{proof}

\noindent Equations in Theorem \ref{theo:A} can be written as a single matrix equation. Let $\vec{B}_{ n}\in\mathbb{M}_{1,n+1}$ be given by 
\[\vec{B}_{ n}=\left[
\p\ci B_{ n}=0\cd, P\ci B_{ n}=1\cd , \p\ci B_{ n}=2\cd,  \ldots  ,\p\ci B_{ n}=n\cd \right],\]
and denote by $\mathcal{M}_{n}\in\mathbb{M}_{n+1,n}$ the matrix such that

\begin{equation}
	\left(\mathcal{M}_{n}\right)_{a,b}=\begin{cases}
		F\left(1-q-\left(n-1-2a\right)\varepsilon\right) & \text{for }a=b\\
		\overline{F}\left(1-q-\left(n-1-2a\right)\varepsilon\right) & \text{for }a=b-1\\
		0 & \text{otherwise }.
	\end{cases}
	\label{matrixrepresentationBn}
\end{equation}

\noindent Using the notation above we note that $\vec{B}_{ n}=\mathcal{M}_{n}\cdot\vec{B}_{n-1}$, therefore
$$\vec{B}_{ n}=\prod_{k=1}^{n}\mathcal{M}_{k}\vec{B}_{0,i},$$
where
$\vec{B}_{0}=\left[F\left(1-q\right),\overline{F}\left(1-q\right)\right]$. This representation is used for some numerical examples in Section 5.
\bigskip

\noindent The explicit distribution of $B_{ n}$ is not easy to obtain even in simple cases (such as the case when $F$ is a uniform distribution). Nevertheless, in the following result we provide a simple asymptotic expression for this distribution.

\begin{theorem}\label{teobonito}
	Let $\left\{p_n(k),k=0,1,\dots,n\right\}$ denote the probability function of a $Binomial(n,p)$ distribution, with $p:=\overline{F}(1-q) $. Suppose $n,\varepsilon$ are such that $n^2\varepsilon\to0$ as $n\to\infty$ and $F$ is absolutely continuous with density $f$ such that $f$ is continuous at $1-q$, then,
	
	$$\lim\limits_{n\to\infty}\frac{\p\ci B_{ n}=k\cd}{p_n(k)}=1,\quad \forall k\in\{0,1,\dots,n\}.$$
\end{theorem}
\begin{proof}
	First we prove the case when $k\notin\{0,n\}$. Following the arguments in the proof of Theorem \ref{theo:A}, we have 
	\begin{align}
		\p\ci B_{ n}=k\cd\nonumber
		&=F\pai1-q-(n-1-2k)\varepsilon\pad\sum\limits_{(x_1,\dots,x_{n-1})\in U_k(n-1)}\p\ci  Y (1)=x_1,\dots,Y (n-1)=x_{n-1}\cd\nonumber \\
		&+\overline{F}\pai1-q-(n-1-2(k-1))\varepsilon\pad\sum\limits_{(x_1,\dots,x_{n-1})\in U_{k-1}(n-1)} \p\ci  Y (1)=x_1,\dots,Y (n-1)=x_{n-1}\cd, \label{probafea2}
	\end{align}
	
	where
	$$U_k(n)=\{(x_1,\dots,x_n)\in \{0,1\}^n:x_1+\dots+x_n=k\}.$$
	
	From the hypothesis $n^2\varepsilon\to 0$ it follows that $\varepsilon\to 0$. Hence, by the assumption of continuity of $F$, the following convergences as $n\to\infty$ hold: 
	
	$$F\pai1-q-(n-1-2k)\varepsilon\pad\to F\pai 1-q\pad,\qquad \overline{F}\pai1-q-(n-1-2(k-1))\varepsilon\pad\to \overline{F}(1-q).$$
	
	Note from the matrix representation given in (\ref{matrixrepresentationBn}) that  
	$\p\ci  Y (1)=x_1,\dots,Y (n-1)=x_{n-1}\cd$ can be expressed as
	
	\begin{equation*}
		p_{n-1,k}:=\prod\limits_{j=1}^{n-1}\overline{F}^{v_1(j)}\Big(1-q-u_1(j)\varepsilon+u_0(j)\varepsilon\Big)F^{1-v_1(j)}\Big(1-q-u_1(j)\varepsilon+u_0(j)\varepsilon\Big),
	\end{equation*}
	
	where 
	
	$$u_1(j)=\sum\limits_{a=1,x_a=1}^{j} (-1)^{x_a},\quad u_0(j)=\sum\limits_{a=1,x_a=0}^{j} (-1)^{x_a},$$
	
	and $v_1(j)=1$ if student understood session $j$.
	
	We have the following bounds for $p_{n-1,k}$:
	
	\begin{equation}\label{cocientefeo1}
		\overline{F}^k\Big(1-q+u_0(j)\varepsilon\Big)F^{n-1-k}\Big(1-q-u_1(j)\varepsilon\Big)\leq p_{n-1,k}\leq\overline{F}^k\Big(1-q-u_1(j)\varepsilon\Big)F^{n-1-k}\Big(1-q+u_0(j)\varepsilon\Big)
	\end{equation}
	
	Since the terms with the tail $\overline{F}$ converge to $\overline{F}(1-q)$ and their exponents do not depend on $n$, we only need to prove that 
	$$\frac{F^{n-1-k}\Big(1-q-u_1(j)\varepsilon\Big)}{F^{n-1-k}(1-q)}\to 1,\quad n\to\infty,$$ 
	
	or equivalently
	$$
	(n-1-k)\log\pai\frac{F\Big(1-q-u_1(j)\varepsilon\Big)}{F(1-q)}\pad \to 0,\quad n\to\infty.
	$$
	
	Using the hypothesis $n^2\varepsilon\to 0$ we may write $\varepsilon=c n^{-{2-\eta}}$ for some $c,\eta>0$. Applying L'Hôpital's rule we obtain
	
	$$     \lim\limits_{n\to\infty}n\log\pai\frac{F\Big(1-q-u_1(j)\varepsilon\Big)}{F(1-q)}\pad
	=cu(j)(2+\eta)\lim\limits_{n\to\infty}\frac{n^{-3-\eta}F(1-q)f(1-q-u_1(j)c n^{-2-\eta})}{(-n^{-2})F(1-q-u_1(j)c n^{-2-\eta})}=0$$
	
	From the limit above and (\ref{cocientefeo1}) it follows that \begin{equation}
		\dfrac{p_{n-1,k}}{\overline{F}^k(1-q)F^{n-1-k}(1-q)}\to 1, n\to\infty.\label{limitea1}
	\end{equation} 
	
	Now let us consider the term
	$$\sum\limits_{(x_1,\dots,x_{n-1})\in U_k(n-1)}\p\ci  Y (1)=x_1,\dots,Y (n-1)=x_{n-1}\cd.$$
	
	Since $|U_k(n-1)|=\binom{n-1}{k}$, using the result in equation (\ref{limitea1}), for an arbitrary $\beta>0$ and sufficiently large $n$ we have
	
	$$    1-\beta\leq \sum\limits_{(x_1,\dots,x_{n-1})\in U_k(n-1)}\frac{\p\ci  Y (1)=x_1,\dots,Y (n-1)=x_{n-1}\cd}{\binom{n-1}{k}\overline{F}^k(1-q)F^{n-1-k}(1-q)}\leq 1+\beta.$$
	
	The result follows by letting $n\to\infty$ and $\beta\to 0$. The result for the second term in equation (\ref{probafea2}) is obtained analogously. Now we proceed in a similar way to prove that $\p\ci B_{ n}=0\cd$ is asymptotically equivalent to $p_n(0)$. It might be easily checked that $\p\ci B_n=0\cd=F(1-q)\prod\limits_{j=1}^{n-1}F(1-q+j\varepsilon)$, hence
	
	$$ 1=\pai\frac{F(1-q)}{F(1-q)}\pad^n\leq \frac{F(1-q)\prod\limits_{j=1}^{n-1}F(1-q+j\varepsilon)}{F^n(1-q)}\leq \pai\frac{F(1-q+n\varepsilon)}{F(1-q)}\pad^{n-1}.$$
	
	Using the representation $\varepsilon=c n^{-\eta-2}$ and L'Hôpital's rule again, we obtain
	
	$$    \lim\limits_{n\to\infty}(n-1)\log \frac{F(1-q+n\varepsilon)}{F(1-q)}=-c(\eta+2)\lim\limits_{n\to\infty}(n-1)^2\frac{n^{-\eta-2}F(1-q+c n^{-\eta-1})}{F(1-q)}f\pai1-q+c n^{-\eta-1}\pad=0.$$
	
	The proof for $\p\ci B_{ n}=n\cd$ is analogous.
\end{proof}
\bigskip

\noindent Theorem \ref{teobonito} says that for a sufficiently large number of sessions, the dependence becomes less relevant, hence the number of sessions that each student understands behaves like a binomial distribution in which the occurrence of successes is independent.

\section{The case of the uniform distribution}

In this section we present some important quantities in the particular case when the students initial distribution for understanding is uniform in $[0,1]$. Throughout this section we assume that $\varepsilon$ is such that $[1-q-(n-1)\varepsilon,1-q+(n-1)\varepsilon]\subseteq[0,1]$, which we name as \textit{Hypothesis 1}.

\begin{proposition}\label{Prop2}
	Let $F$ be the uniform distribution over $(0,1)$ and assume Hypothesis 1 holds.
	Then for $i\leq n$ we have
	$$
	\overline{F}_{i}\left(1-q-m\varepsilon\right)=\left(1+2\varepsilon\right)^{i-1}\left\{ q-\frac{1}{2}\right\} +\frac{1}{2}+m\varepsilon
	$$
\end{proposition}
\begin{proof}
	Recall that $F_1=F$. The result holds for $i=1$, since
	$$
	\overline{F}_{1}\left(1-q-m\varepsilon\right)=q+m\varepsilon
	$$
	
	We proceed by induction, assuming that for an integer $j\geq1$
	
	$$
	\overline{F}_{j}\left(1-q-m\varepsilon\right)=\left(1+2\varepsilon\right)^{j-1}\left\{ q-\frac{1}{2}\right\} +\frac{1}{2}+m\varepsilon
	$$
	
	By the Law of Total Probability and induction hypothesis
	
	\begin{align*}
		\overline{F}_{j+1}\left(1-q-m\varepsilon\right) & =\overline{F}_{j}\left(1-q-\left(m+1\right)\varepsilon\right)\overline{F}_{j}\left(1-q\right)+\overline{F}_{j}\left(1-q-\left(m-1\right)\varepsilon\right)\left(1-\overline{F}_{j}\left(1-q\right)\right)\\
		& =\left[\left(1+2\varepsilon\right)^{j-1}\left\{ q-\frac{1}{2}\right\} +\frac{1}{2}+\left(m+1\right)\varepsilon\right]\left[\left(1+2\varepsilon\right)^{j-1}\left\{ q-\frac{1}{2}\right\} +\frac{1}{2}\right]\\
		& +\left[\left(1+2\varepsilon\right)^{j-1}\left\{ q-\frac{1}{2}\right\} +\frac{1}{2}+\left(m-1\right)\varepsilon\right]\left[\frac{1}{2}-\left(1+2\varepsilon\right)^{j-1}\left\{ q-\frac{1}{2}\right\} \right]\\
		& =\frac{1}{2}\left(1+2\varepsilon\right)^{j-1}\left\{ q-\frac{1}{2}\right\} +\frac{1}{2}\left[\frac{1}{2}+\left(m+1\right)\varepsilon\right]+\left[\left(1+2\varepsilon\right)^{j-1}\left\{ q-\frac{1}{2}\right\} \right]\left[\frac{1}{2}+\left(m+1\right)\varepsilon\right]\\
		& +\frac{1}{2}\left(1+2\varepsilon\right)^{j-1}\left\{ q-\frac{1}{2}\right\} +\frac{1}{2}\left[\frac{1}{2}+\left(m-1\right)\varepsilon\right]-\left[\left(1+2\varepsilon\right)^{j-1}\left\{ q-\frac{1}{2}\right\} \right]\left[\frac{1}{2}+\left(m-1\right)\varepsilon\right]\\
		& =\left(1+2\varepsilon\right)^{j-1}\left\{ q-\frac{1}{2}\right\} +\frac{1}{2}+m\varepsilon+2\varepsilon\left[\left(1+2\varepsilon\right)^{j-1}\left\{ q-\frac{1}{2}\right\} \right]\\
		& =\left(1+2\varepsilon\right)^{j}\left\{ q-\frac{1}{2}\right\} +\frac{1}{2}+m\varepsilon
	\end{align*}
	
	Hence the result follows.
\end{proof}

\begin{theorem}
	Let $F$ be the uniform distribution over $(0,1)$ and assume
	Hypothesis 1 holds. Then for $m\leq n$ we have
	$$
	p_{m}=\frac{1}{2}+\left(1+2\varepsilon\right)^{m-1}\left(q-\frac{1}{2}\right)
	$$
\end{theorem}
\begin{proof}
	For $1\leq j\leq m-1$, it follows from Proposition \ref{Prop2}
	that
	$$
	\overline{F}_{j}\left(1-q-\varepsilon\right)-\overline{F}_{j}\left(1-q+\varepsilon\right)=2\varepsilon
	$$
	
	Since $\overline{F}\left(1-q\right)=q$, by Theorem \ref{Teo1} we have
	
	\small
	\begin{align}
		p_{m} & =\overline{F}_{1}\left(1-q\right)\prod_{j=1}^{m-1}\left[\overline{F}_{j}\left(1-q-\varepsilon\right)-\overline{F}_{j}\left(1-q+\varepsilon\right)\right]\nonumber +\sum_{i=1}^{m-1}\overline{F}_{i}\left(1-q+\varepsilon\right)\prod_{j=i+1}^{m-1}\left[\overline{F}_{j}\left(1-q-\varepsilon\right)-\overline{F}_{j}\left(1-q+\varepsilon\right)\right]\nonumber \\
		& =q\left[2\varepsilon\right]^{m-1}+\left\{ \frac{1}{2}-\varepsilon\right\} \sum_{i=1}^{m-1}\left(2\varepsilon\right)^{m-1}\left(2\varepsilon\right)^{-i}+\left\{ q-\frac{1}{2}\right\} \sum_{i=1}^{m-1}\left(1+2\varepsilon\right)^{i-1}\left[2\varepsilon\right]^{m-1-i}\nonumber \\
		& =\left[2\varepsilon\right]^{m-1}\left[q+\left\{ \frac{1}{2}-\varepsilon\right\} \sum_{i=1}^{m-1}\left(2\varepsilon\right)^{-i}+\frac{q-\frac{1}{2}}{1+2\varepsilon}\sum_{i=1}^{m-1}\left(\frac{1+2\varepsilon}{2\varepsilon}\right)^{i}\right].\label{eq:labuena}
	\end{align}
	\normalsize
	
	Using the identities
	$$\sum_{i=1}^{n-1}\left(\frac{1+2\varepsilon}{2\varepsilon}\right)^{i}=\left[1+2\varepsilon\right]\left[\left(\frac{1+2\varepsilon}{2\varepsilon}\right)^{n-1}-1\right]\text{ and }
	\sum_{i=1}^{n-1}\left(2\varepsilon\right)^{-i}=\left[\frac{1}{1-2\varepsilon}\right]\left[\left(\frac{1}{2\varepsilon}\right)^{n-1}-1\right],$$
	
	the right hand of equation (\ref{eq:labuena}) becomes
	
	\begin{align*}
		& \left[2\varepsilon\right]^{m-1}\left[q+\left\{ \frac{1}{2}-\varepsilon\right\} \left[\frac{1}{1-2\varepsilon}\right]\left[\left(\frac{1}{2\varepsilon}\right)^{m-1}-1\right]+\frac{q-\frac{1}{2}}{1+2\varepsilon}\left[1+2\varepsilon\right]\left[\left(\frac{1+2\varepsilon}{2\varepsilon}\right)^{m-1}-1\right]\right]\\
		& =\left[2\varepsilon\right]^{m-1}\left[\frac{1}{2}\left(\frac{1}{2\varepsilon}\right)^{m-1}+q\left(\frac{1+2\varepsilon}{2\varepsilon}\right)^{m-1}-\frac{1}{2}\left(\frac{1+2\varepsilon}{2\varepsilon}\right)^{m-1}\right]\\
		& =\left[\frac{1}{2}+q\left(1+2\varepsilon\right)^{m-1}-\frac{1}{2}\left(1+2\varepsilon\right)^{m-1}\right].
	\end{align*}
	
	Hence we obtain
	
	$$p_{m}=\frac{1}{2}+\left(1+2\varepsilon\right)^{m-1}\left(q-\frac{1}{2}\right).$$
	
	And the result follows.
\end{proof}

\begin{center} 
	\begin{table}
		\begin{tabular}{c c}
			\multicolumn{1}{r}{\includegraphics[scale = 0.4]{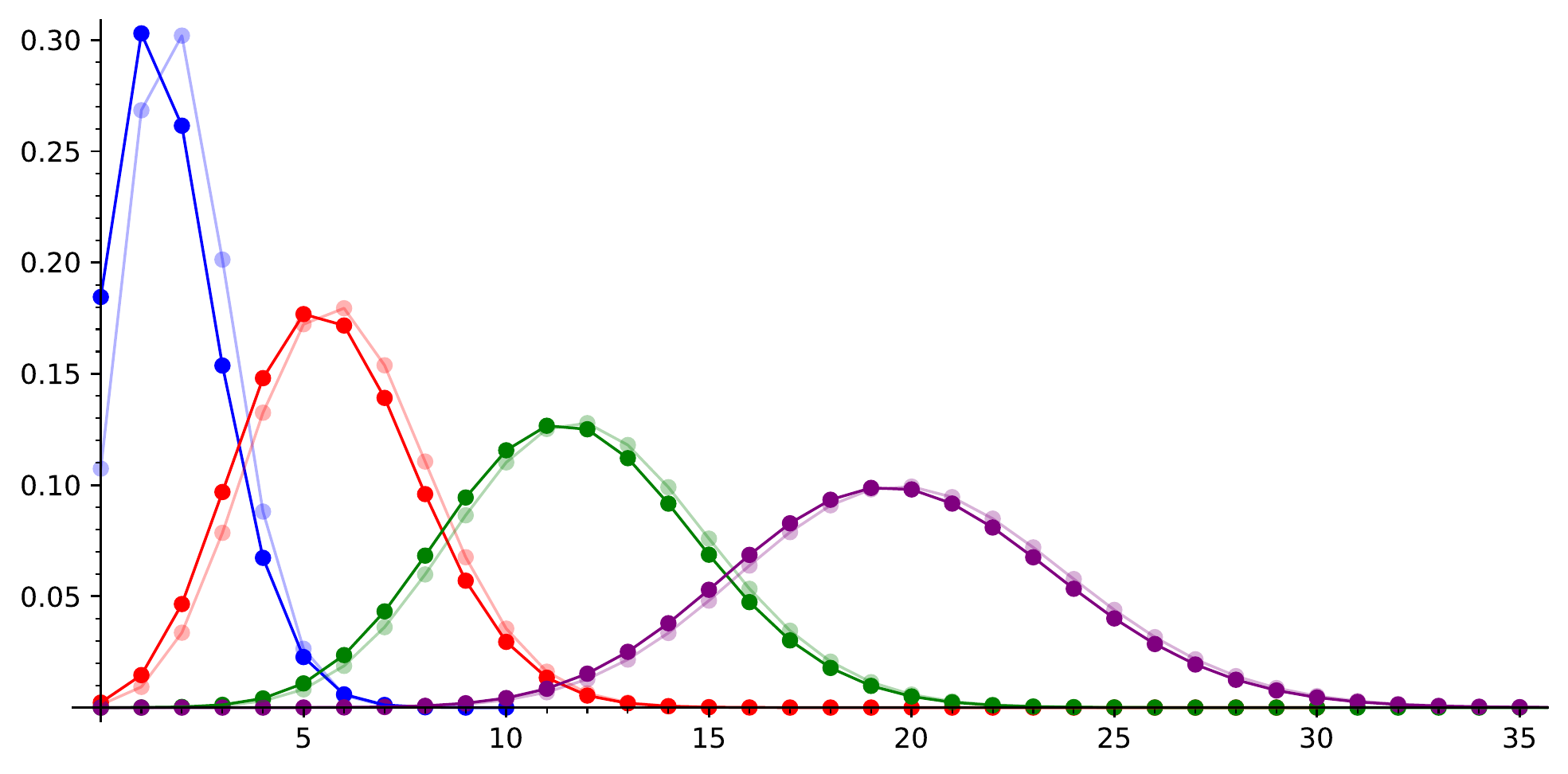}} & \includegraphics[scale = 0.4]{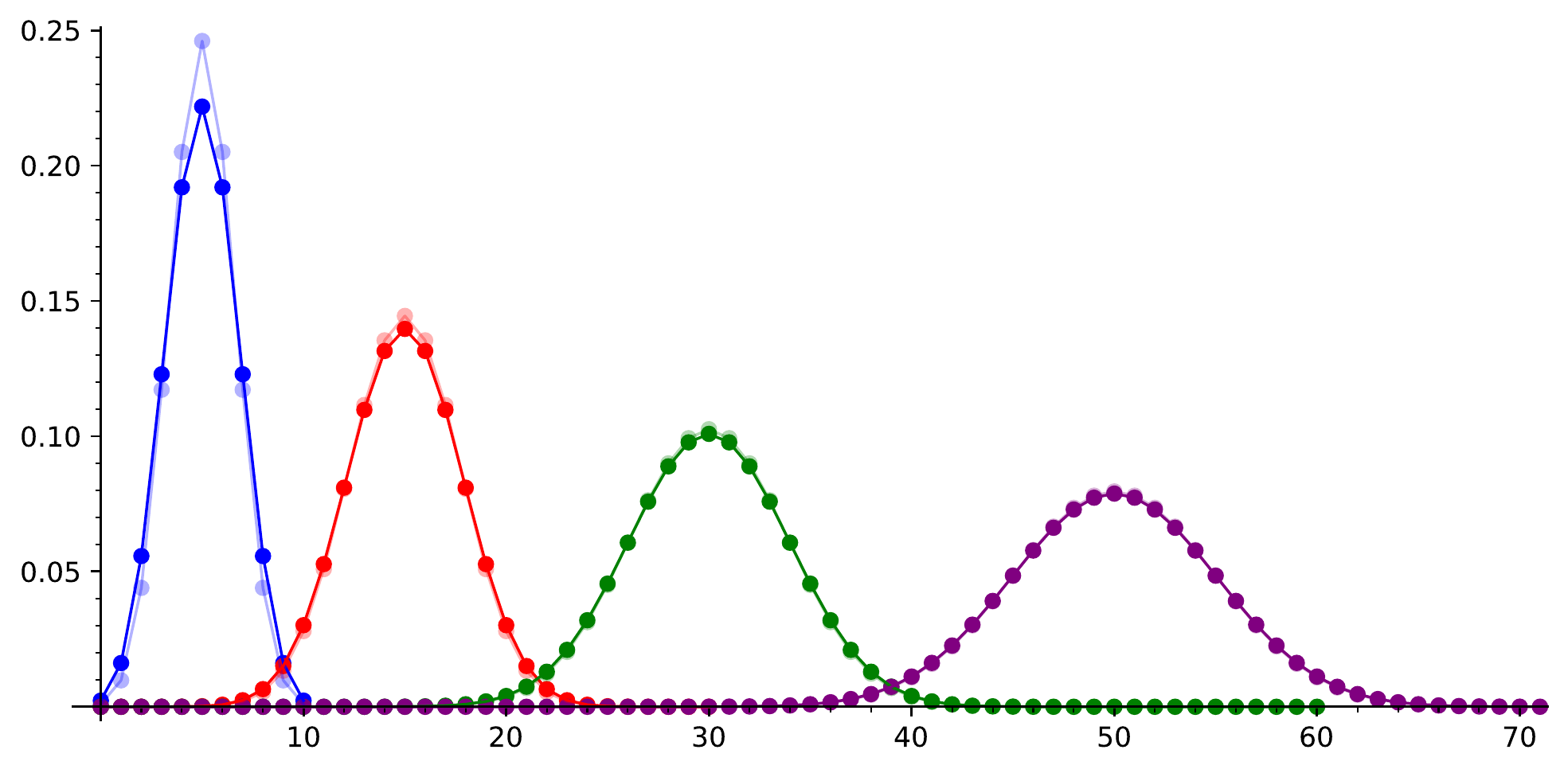} \\
			(a)                     & (b) \\
			\multicolumn{2}{c}{\includegraphics[scale = 0.4]{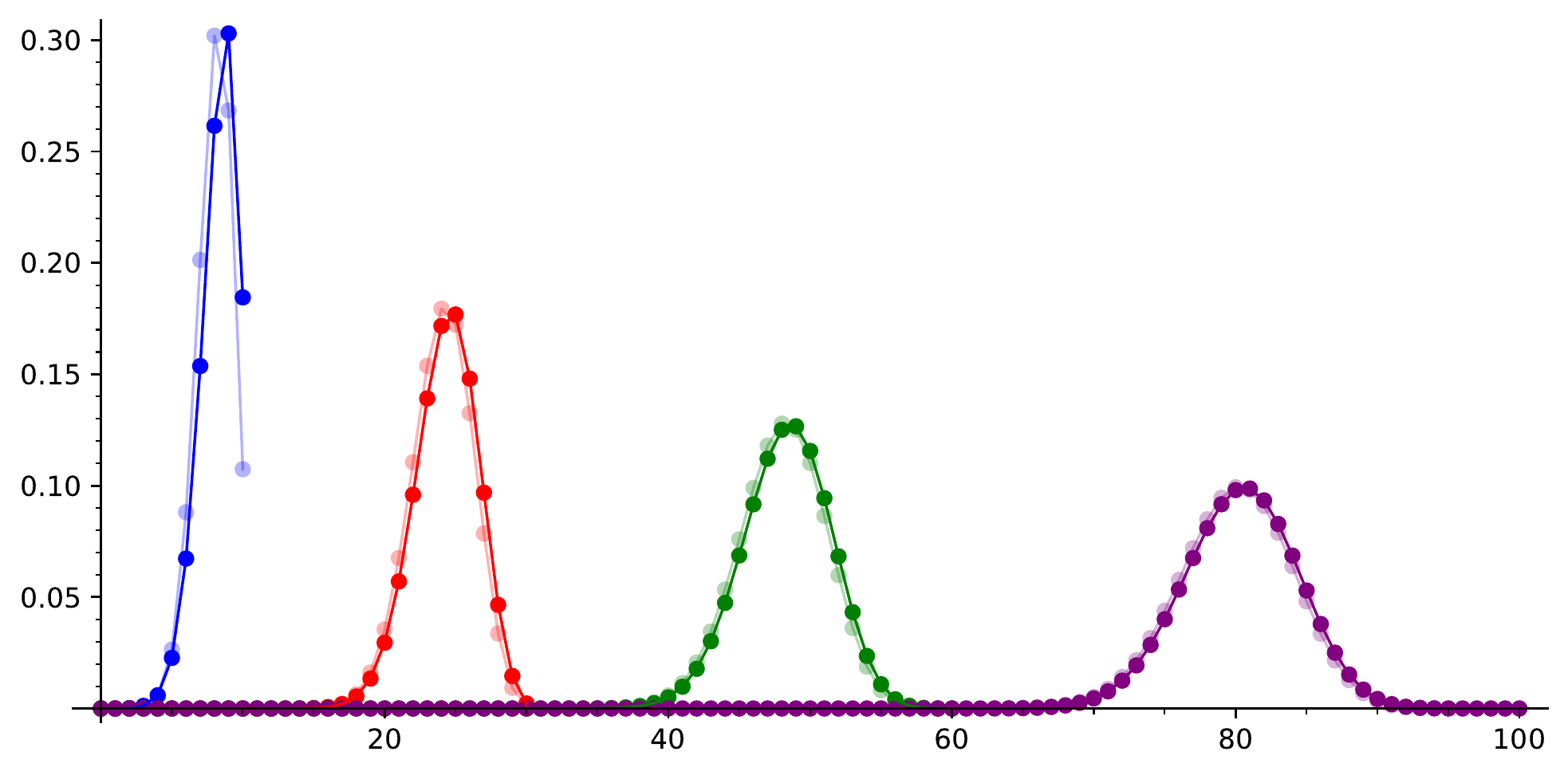}}     \\
			\multicolumn{2}{c}{(c)}    
		\end{tabular}
		\caption{ In this figure we plot some simulations of $B_n$ (dark line) and its approximating binomial distribution (light line).In all the plots we consider $\varepsilon = 1/n^2$ and  $n=10$ in blue; $n=30$ in red; $n=60$ in green and $n=100$ in purple.  Plot (a) is made for $q=0.2$;  Plot (b) is made for $q=0.5$ and Plot (c) for $q=0.8$. } 
		\label{table:1} 
	\end{table}
\end{center}

As we previously mentioned, an analytic expression for the exact distribution of $B_n$ is not easy to obtain even in this simple case when $F$ is the uniform distribution. However, in Theorem \ref{teobonito} we have seen that $B_n$ behaves asymptotically like a binomial random variable. We present some numerical examples of the performance of this approximation, using different values of $n$, $q$ and $\varepsilon$. These examples show that the two distributions are very close even for values of $n$ which are not so large.

In Figure 1 we fix the value of $q$, we simulate the exact distribution and compare it to the approximating binomial distribution. This is made for $n=10, 30, 60, 100$ (plots in blue, red, green and purple respectively). The light line corresponds to the approximating binomial distribution while the dark line represents the simulated exact distribution of $B_n$.

As it can be seen in the distinct plots, the convergence to the binomial distribution is quite fast and it grows faster when $q=.5$. Moreover, it is seen numerically that for any $q$, when $n\geq 40$, $\frac{\p\ci B_{ n}=k\cd}{p_n(k)} \geq 0.95$. 

Our numerical examples show that the speed of convergence depends of the value of $q$, as it can be seen in Figure 2, where we use different values of $q$ with fixed $n$ and compare the exact distribution of $B_n$ to the approximating binomial distribution.

Another point worth mentioning is that the distributions behave symmetrically with respect to $q=0.5$. In this case, the convergence to the binomial distribution is faster than in the other cases. In fact, the cases when $q$ is nearly 0 or 1, present a slower convergence to the binomial distribution. This may imply that the dependence is stronger when the quality of the class is low or high.

\begin{center} 
	\begin{table} 
		\begin{tabular}{c c} 
			\includegraphics[scale = 0.4]{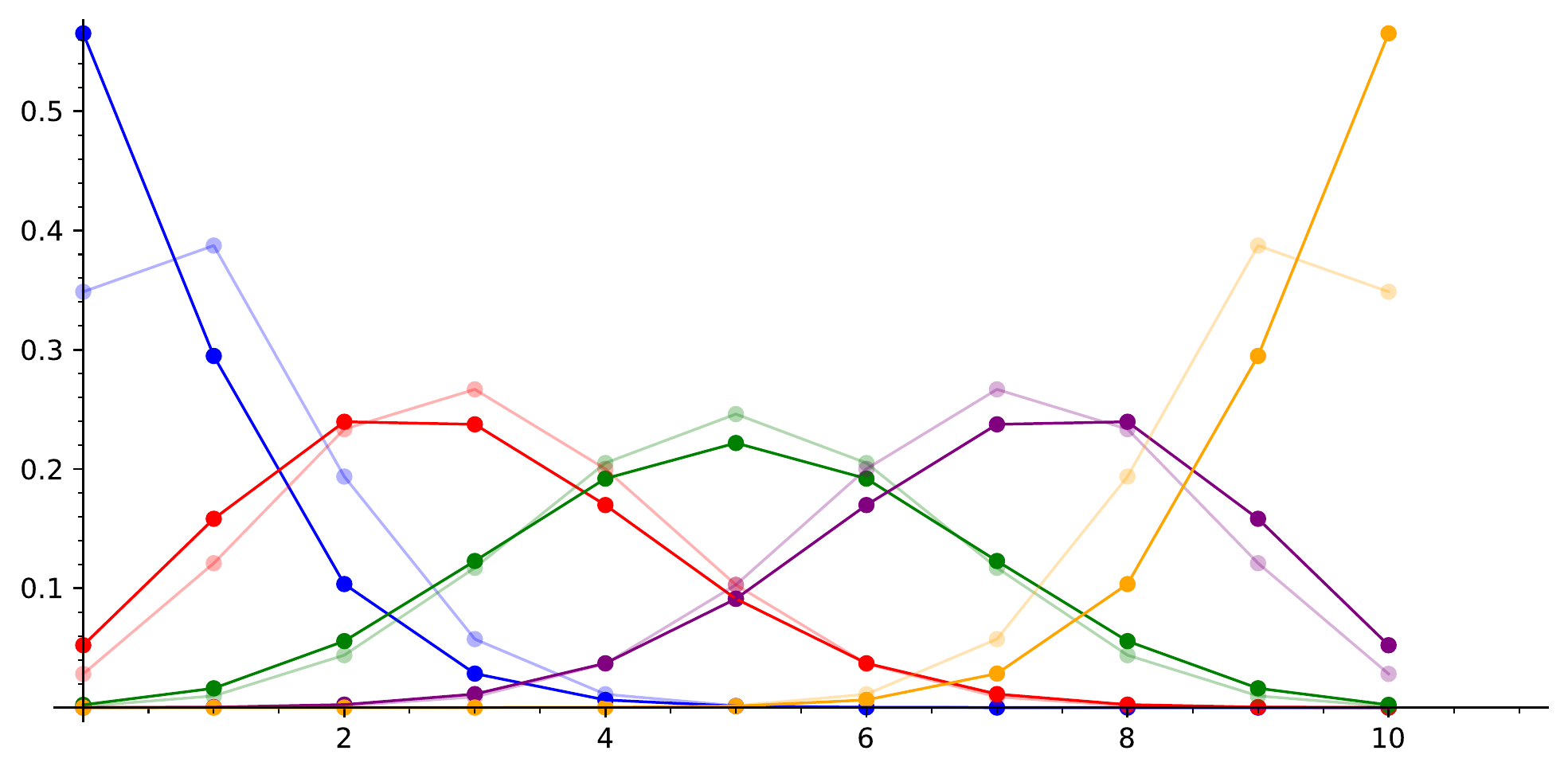} & \includegraphics[scale = 0.4]{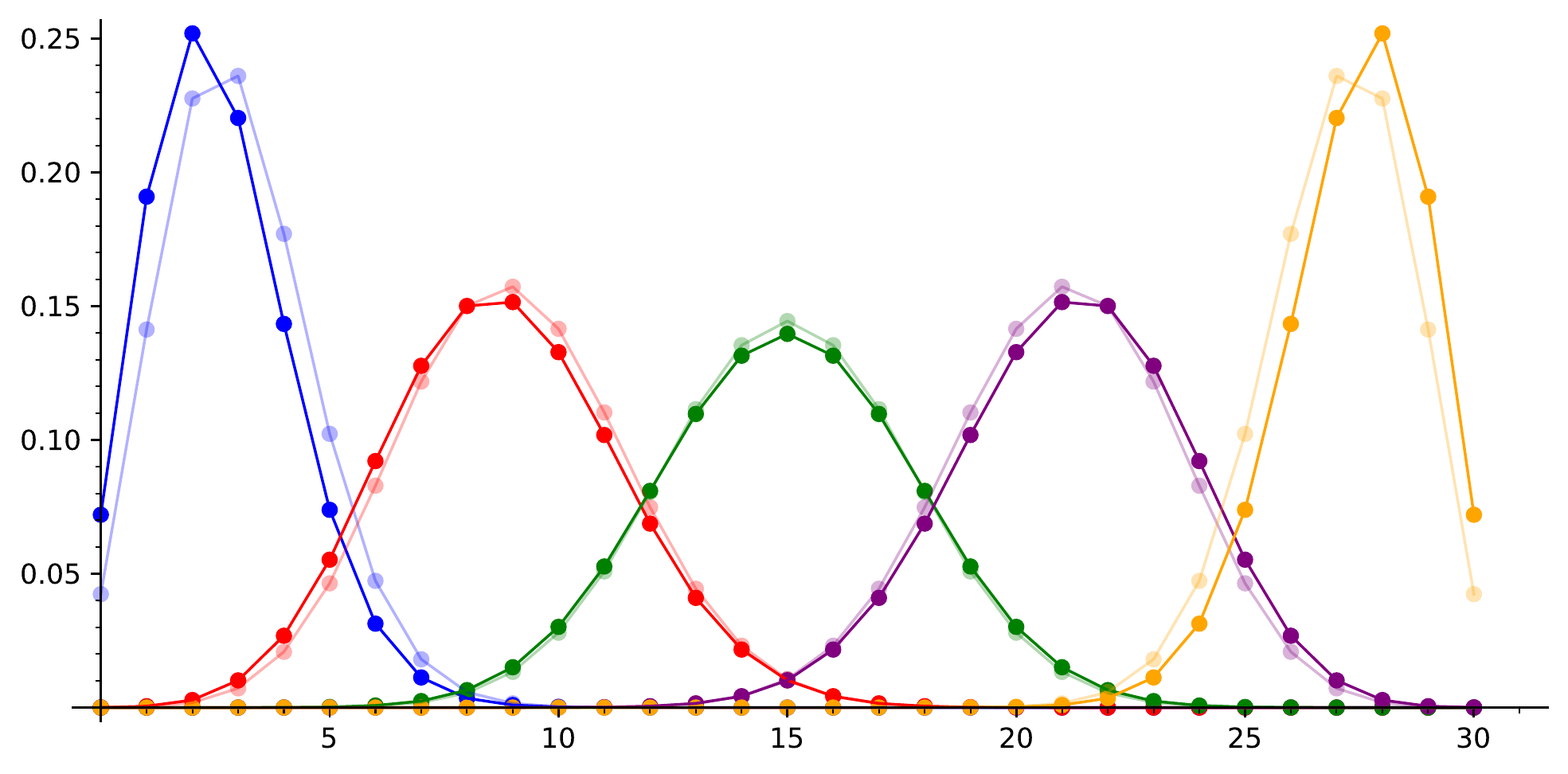} \\
			$(a)$&$(b)$\\
			\includegraphics[scale = 0.4]{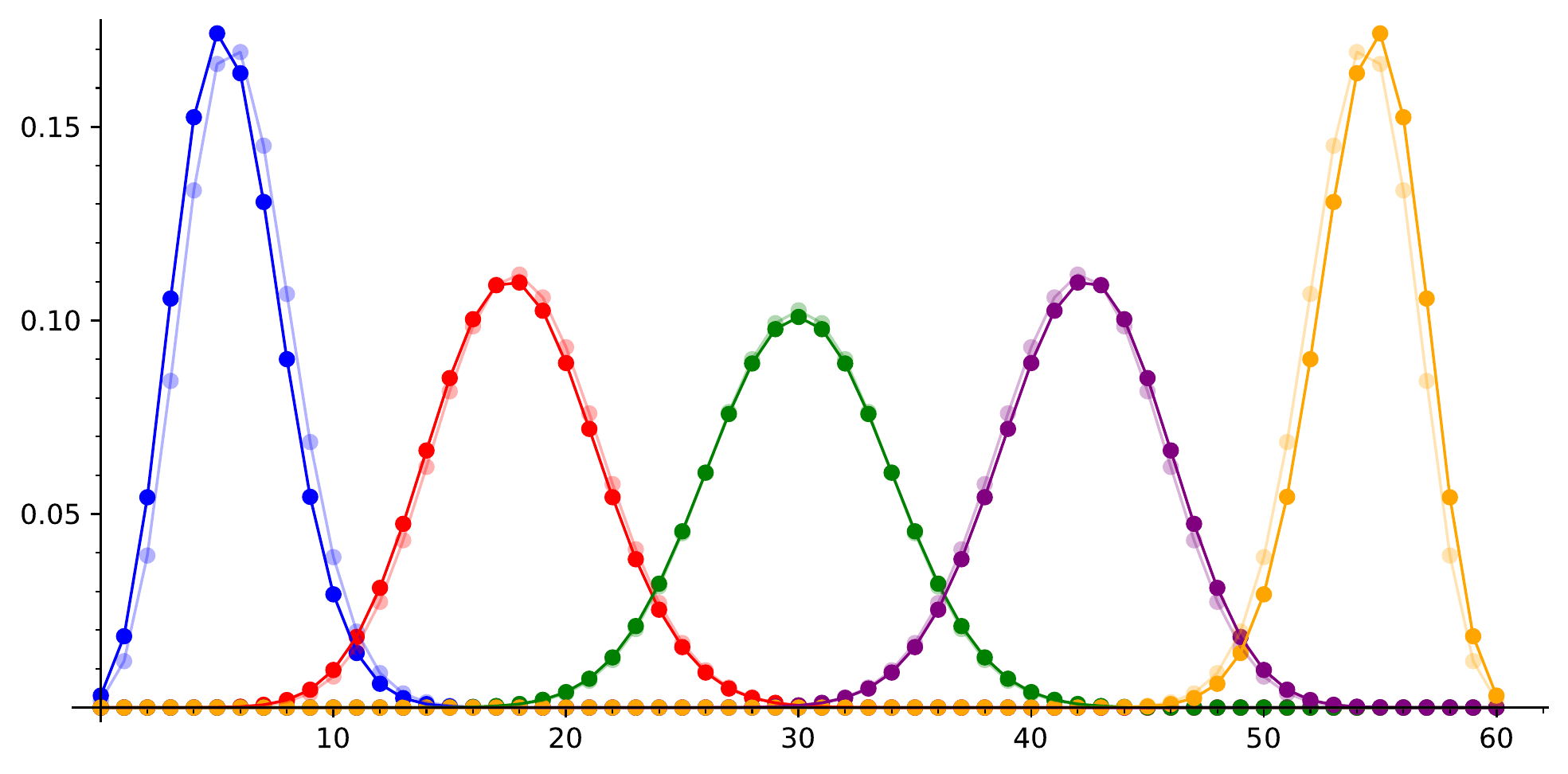} & \includegraphics[scale = 0.4]{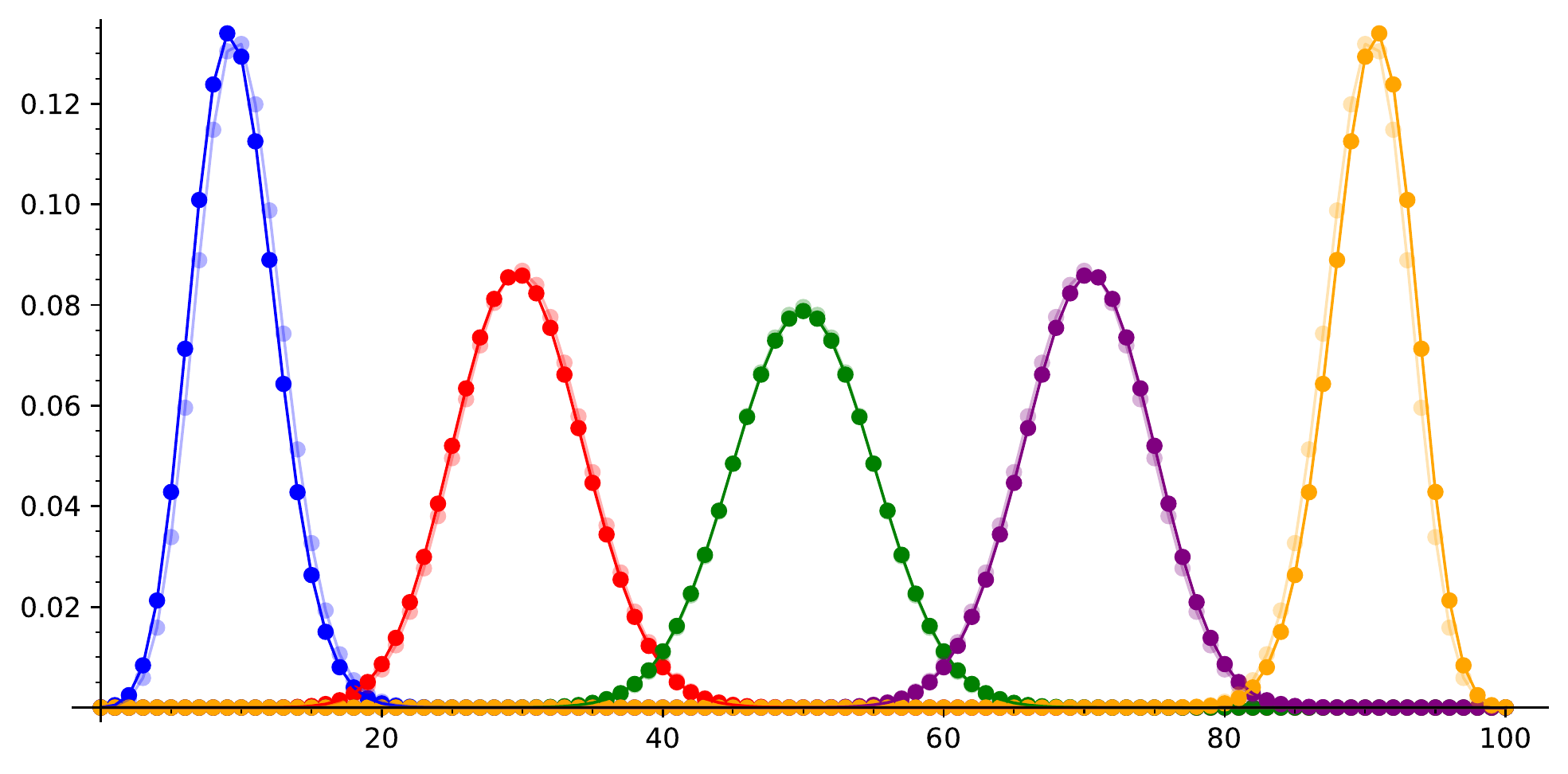} \\
			$(c)$ & $(d)$
		\end{tabular} 
		\caption{In this figure we plot some simulations of $B_n$ (dark line) and its approximating binomial distribution (light line).In all the plots we consider $\varepsilon = 1/n^2$ and  $q=0.1$ in blue; $q=0.3$ in red; $q=0.5$ in green; $q=0.7$ in purple and $q=0.9$ in yellow.  Plot (a) is made for $n=10$;  Plot (b) is made for $n=30$; Plot (c) is made for $n=60$ and Plot (d) for $n=100$.}
		\label{table:2} 
	\end{table} 
\end{center}

\section{Conclusions and Future work}
In this work we present a model to study the behaviour of the understanding of a student along several sessions of a course taught totally online and with no interaction between other students. In particular we study the case where the dependence between sessions is relatively small compared to the total number of sessions, as in seminars or panoramic courses. We obtained a recursive expression for the distribution of the number of sessions that the student understands along the course and showed that when the dependence parameter is small, this distribution has a binomial approximation. The speed of convergence of the approximation depends on the number of sessions and the quality of them. Even though this is a simple model, it can be fruitfully extended to consider many more situations, some of which we list below:

\begin{itemize}
	\item The environment we considered assumes the student has no interaction with their classmates. However several studies have shown that collaborative learning provides better results for students. It would be very interesting to modify the model to consider this situation.
	\item We studied the case when $\varepsilon$ is constant and relatively small compared to $n$, but this may not always be the case, as in some science classes. Therefore it would be useful to consider the cases when $\varepsilon$ changes according to the sessions themselves or according to the number of sessions previously understood.
	
	\item All the results obtained in this work were made for $q$ constant, but the value of $q$ can vary along the course due to exhaustion and motivation of the student and the teacher.
	
	\item Even though the distribution of $B_n$ can be approximated by a binomial distribution, we were not able to provide similar results for the corresponding mean and variance.

	\item It is interesting to test the model with real world data and study or develop some statistical procedures in order for the model to be fitted and validated.

\end{itemize}

\vspace{.5cm}
\vspace{.5cm}

\bibliography{biblioising}
\bibliographystyle{abbrv}

\end{document}